\documentclass[12pt, twoside]{article}
\usepackage{amsmath,amsthm,amssymb}
\usepackage{times}
\usepackage{enumerate}

\def\titlerunning#1{\gdef\titrun{#1}}
\makeatletter
\def\author#1{\gdef\autrun{\def\and{\unskip, }#1}\gdef\@author{#1}}
\def\address#1{{\def\and{\\\hspace*{18pt}}\renewcommand{\thefootnote}{}%
\footnote {#1}}%
\markboth{\autrun}{\titrun}}
\makeatother
\def\email#1{e-mail: #1}
\def\subjclass#1{{\renewcommand{\thefootnote}{}%
\footnote{\emph{Mathematics Subject Classification (2010):} #1}}}
\def\keywords#1{\par\medskip
\noindent\textbf{Keywords.} #1}

\newtheorem{theorem}{Theorem}[section]
\newtheorem{corollary}[theorem]{Corollary}
\newtheorem{definition}[theorem]{Definition}

\newtheorem{lemma}[theorem]{Lemma}

\newtheorem{proposition}[theorem]{Proposition}
\newtheorem{remark}[theorem]{Remark}

\numberwithin{equation}{section}

\newcommand{\cala}{{\mathcal{A}}}
\newcommand{\calb}{{\mathcal{B}}}
\newcommand{\cald}{{\mathcal{D}}}

\newcommand{\calh}{\mathcal{H}}
\newcommand{\calo}{\mathcal{O}}
\newcommand{\calf}{\mathcal{F}}
\newcommand{\calt}{\mathcal{T}}
\newcommand{\calv}{{\mathcal{V}}}
\newcommand{\E}{\mathbb{E}}
\newcommand{\rr}{\mathbb{R}}

\newcommand{\vsp}{\vspace*{1,5mm}\\ }
\newcommand{\n}{\noindent }
\newcommand{\g}{{\gamma}}

\newcommand\lbb{\lambda}
\newcommand{\pas}{\mathbb{P}\mbox{-a.s.}}
\newcommand\pp{\partial}
\newcommand\dd{\displaystyle}
\def\a{\alpha}

\def\vp{\varepsilon}
\def\wt{\widetilde}
\def\barr{\begin{array}}
\def\earr{\end{array}}

\begin{document}


\baselineskip=17pt


\titlerunning{Stochastic partial differential equations driven by linear multiplicative~noise}

\title{An operatorial approach to  stochastic partial differential equations driven by linear multiplicative~noise}

\author{Viorel Barbu
\and
Michael R\"ockner}

\date{}

\maketitle

\address{V. Barbu: Octav Mayer Institute of Mathematics (Romanian Academy)   and Al.I. Cuza University of Ia\c si, Romania; \email{vbarbu41@gmail.com}
\and
M. R\"ockner: Fakult\"at f\"ur Mathematik, Universit\"at Bielefeld,  D-33501 Bielefeld, Germany; \email{roeckner@math.uni-bielefeld.de}}

\subjclass{Primary 60H15; Secondary 47H05; 47J05}


\begin{abstract} In this paper, we develop a new general approach to the existence and uniqueness theory of infinite dimensional stochastic equations of the form $$dX+A(t)X dt=X dW\mbox{ in }(0,T)\times H,$$ where $A(t)$ is a nonlinear monotone and demicontinuous ope\-ra\-tor   from $V$ to $V'$,  coercive and with polynomial growth. Here,   $V$ is a reflexive   Banach space con\-ti\-nuously and densely embedded in a Hilbert space $H$ of (generalized) functions on a domain $\calo\subset\rr^d$  and $V'$ is the dual of $V$ in the duality induced by $H$ as pivot space.  Furthermore, $W$ is a  Wiener process in $H$.   The  new approach is based on an operatorial reformulation  of the stochastic equation which is quite robust under perturbation of $A(t)$. This leads to new  existence and uniqueness results of a larger class of equations with linear multiplicative noise than the one treatable by the known approaches. In addition, we obtain regularity results for the solutions with respect to both the time and spatial variable which are sharper than the classical ones. New applications include stochastic partial differential equations, as e.g. stochastic transport  equations.

\keywords{Maximal monotone operator, stochastic integral, operatorial equations.}
\end{abstract}

\section{Introduction}

 This work is concerned with a new general functional approach to the exis\-tence and uniqueness theory for nonlinear stochastic infinite dimensional equations with monotone and demicontinuous time dependent nonlinearities from a reflexive Banach  space to its dual. Theorem \ref{t2.1} is the main general result obtained in this way, but the main point is the new method for its proof, namely rewriting the stochastic equation in operatorial form, which in turn is quite robust under perturbation, as formulated in Proposition \ref{p2}. The latter then  applies to a larger variety of examples, which do not \mbox{exactly} fit the general framework of Theorem \ref{t2.1}, but can be treated by a direct approach on the basis of Proposition \ref{p2}   (see, e.g., Section 6.3). In the li\-terature  on infinite dimensional stochastic differential equations, the type of equations in Theorem \ref{t2.1} was studied firstly by the classical Galerkin method, combined with monotonicity arguments by Pardoux \cite{1} and de\-ve\-loped later   in a general setting by Krylov and Rozovski \cite{2}. (A detailed presentation of these results is given in the monograph \cite{3}.) The approach we are pro\-po\-sing here is,  principally,  different and covers  more general types of nonlinear stochastic PDEs, as long as the noise is linear multiplicative. Moreover, the   results one obtains in this way are sharper, in regard to  new regularity properties of the solutions. In a few words, it consists in representing, via a rescaling transformation, the stochastic initial value problem as a random operator equation of monotone type in a new convenient space of stochastic processes and in invoking the standard perturbation theory for nonlinear maximal monotone operators to  get existence and uniqueness of solutions. In the special case, where the nonlinear operator is the subgradient of a convex function, the problem reduces to a convex optimization problem.

\section{Preliminaries}

Here we consider   the stochastic differential equation
\begin{equation}\label{e1.1}
\begin{array}{l}
dX(t)+A(t)X(t)dt= X(t)dW(t),\ t\in(0,T),\vsp
X(0)=x,\end{array}\end{equation}in a real separable Hilbert space $H$, whose elements are functions or distributions on a bounded and open set $\calo\subset\rr^d$ with smooth boundary $\partial \calo$. In particular, $H$ can be any  of the spaces $L^2(\calo)$, $H^1_0(\calo)$, $H^{-1}(\calo),$ $H^k(\calo)$, $k=1,...$, with the corresponding Hilbertian structure.  Here $H^1_0(\calo),$ $H^k(\calo)$ are the standard $L^2$-Sobolev spaces on $\calo$, and
$W$ is a Wiener   process of the form
\begin{equation}\label{e1.2}
W(t,\xi)=\sum^{\infty}_{j=1}\mu_je_j(\xi)\beta_j(t),\ \xi\in\calo,\  t\ge0,\end{equation}
where $\{\beta_j\}^\infty_{j=1}$ is an  independent system of real-valued  Brownian motions on a probability space $\{{\Omega},\calf,\mathbb{P}\}$ with   natural filtration $(\calf_t)_{t\ge0}$. Here,  $e_j\in C^2(\overline{\calo})\cap H$ is an orthonormal basis in $H$, $\mu_j\in\rr$, $j=1,2,...$.

The following hypotheses will be   in effect throughout this work.
\begin{itemize}
\item[(i)] {\it There is a reflexive Banach space $V$ with dual $V'$ such that $V\subset H$, con\-ti\-nuously and densely. Hence $V\subset H$ $(\equiv H')\subset V'$ continuously and densely. $($Note that this implies that also $V$ is separable.$)$  Moreover, $V$~and $V '$ are strictly convex $($which can always be achieved by con\-si\-de\-ring an appropriate equivalent norm on $V$ by Asplund's Theorem, see {\rm\cite[Theorem 1.2, p.2]{8})}.}

    \item[(ii)] {\it $A:[0,T]\times V\times{\Omega}\to V'$ is progressively measurable, i.e, for every $t\in[0,T]$, this operator restricted to $[0,t]\times V\times{\Omega}$ is $\calb([0,t])\otimes \calb(V)\otimes\calf_t$ measurable.}

    \item[(iii)] {\it There is $\delta\ge0$ such that, for each $t\in[0,T]$,  ${\omega}\in{\Omega}$, the operator $u\to \delta u+A(t,\omega)u$ is monotone and demicontinuous $($that is, strongly-weakly continuous$)$ from $V$ to $V'$. Moreover, there are $1<p<\infty$, $\alpha_i$  and $\g_i\in\rr$, $\alpha_1>0,$ $i=1,2,3,$ such that, $\pas$,
        \begin{eqnarray}
        \left<A(t,\omega)u,u\right>&\ge&\alpha_1|u|^p_V
        +\alpha_2|u|^2_H+\alpha_3,\ \forall  u\in V,\ t\in[0,T],\label{e1.3}\\[2mm]
        |A(t,\omega)u|_{V'}&\le&\g_1|u|^{p-1}_V
        +\g_2+\g_3|u|_{H},\ \forall  u\in V,\ t\in[0,T].\label{e1.4}
        \end{eqnarray}}
        \item[(iv)] {\it $e^{\pm W(t)}$ is, for each $t$, a multiplier in $V$ and a symmetric multiplier in $H$, such\break that there exists an $(\calf_t)$-adapted, $\rr_+$-valued process $Z(t)$, $t\in[0,T]$, with\break $\E\left[\sup\limits_{t\in[0,T]}
            |Z(t)|^r\right]<\infty$, $\forall r\in[1,\infty)$   and such that $\pas$
            \begin{equation}
            \label{e2.4a}
            \barr{ll}
            |e^{\pm W(t)}y|_V\le Z(t)|y|_V,& \forall t\in[0,T],\ \forall y\in
            V,\vsp
            |e^{\pm W(t)}y|_H\le Z(t)|y|_H,&\forall t\in[0,T],\ \forall y\in H.\earr\end{equation}

            Furthermore, we assume that, $\pas$,
           \begin{equation}\label{e25prim}
           \barr{l}
            \left< e^{\pm W(t)}x,y\right>=
            \left<x,e^{\pm W(t)}y\right>,\ \forall x,y\in H,\ t\in[0,T],\vsp
            t\longmapsto e^{\pm W(t)}\in H\mbox{ is continuous}. \earr\end{equation}}
\end{itemize}

  We also note  that,  by Fernique's theorem,
\begin{equation}\label{e2.5a}
\exp\left(\sup_{0\le t\le T}|W(t)|_\infty\right)\in L^q(\Omega),\ \forall q\in(0,\infty).\end{equation}
which will be used to estimate $\E|e^{W(t)}|^q_V$ and $\E|e^{W(t)}|^q_H$ to verify (iv) in many situations, where $V$ is a subspace of $L^q$, $1<q<\infty$, or a Sobolev space on $\calo\subset\rr^d$. Above and below, $|\cdot|_V$ and $|\cdot|_{V'}$ denote the norms of $V$ and $V'$, respectively,  and  $\left<\cdot,\cdot\right>$ the duality pairing between $V$ and $V'$; on $H\times H$, $\left<\cdot,\cdot\right>$ is just the scalar product of $H$. The norm of $H$ is denoted by $|\cdot|_H$, and $\calb(H),\calb(V)$ etc. are used to denote the class of Borelian sets on the corresponding spaces.

As regards the basis $\{e_j\}^\infty_{j=1}$ in \eqref{e1.2}, we assume that there exist $\wt\gamma_j\in[1,\infty)$  such that
\begin{equation}\label{e4ay}
\begin{array}{l}
|ye_j|_H\le\wt\gamma_j|e_j|_\infty|y|_H,\ \forall y\in H,\ j=1,2,...,\
\nu:=\displaystyle\sum^\infty_{j=1}\mu^2_j
\wt\gamma^2_j|e_j|^2_\infty<\infty,\end{array}\end{equation}
and, for
\begin{equation}\label{e1.9}
\mu:=\frac12\sum^\infty_{j=1}\mu^2_je^2_j,\end{equation}we  assume  that $\mu$ is a multiplier in    $V$ and a symmetric multiplier in $H$. All these assumptions on $W$, as well as Hypothesis (iv), typically hold in applications, as we will see in the examples in Section 6.

For $y\in H$, we define the operator
$$\sigma(y)h=\sum^\infty_{j=1}\mu_jy\left<h,e_j\right> e_j,\ h\in H.$$Then, by \eqref{e4ay}, we see that $\sigma$ is linear continuous form $H$ to the space $L_2(H)$ of all Hilbert-Schmidt operators on $H$ and $X\,dW=\sigma(X)d\wt W$ in the notation of e.g. \cite{17}, \cite{3} with $\wt W$ being the cylindrical Wiener process on $H$, informally written as $\wt W(t)=\sum\limits^\infty_{j=1}\beta_j(t)e_j.$

As usual, we set $p':=\frac p{p-1},\ 1<p<{\infty}.$

\begin{definition}\label{d1.1} \rm By a {\it  solution} to \eqref{e1.1} for $x\in H$, we mean an $(\calf_t)_{t\ge0}$-adapted process \mbox{$X:[0,T]\to H$} with continuous sample paths,  which sa\-tisfies
\begin{eqnarray}
\label{e1.3a}
&X\in L^{\infty}(0,T;L^2({\Omega};H))\\[1mm]
\label{e1.4a}
&X(t)+\displaystyle \int^t_0A(s)X(s)ds
=x+\displaystyle \int^t_0X(s)dW(s),\ t\in[0,T], \\[1mm]
\label{e1.5}
&A X\in L^{p'}((0,T)\times{\Omega};V'),\ X\in L^p((0,T)\times{\Omega};V).
\end{eqnarray}
\end{definition}

The stochastic integral arising in \eqref{e1.4a} is considered in It\^o's sense.

Under Hypotheses (i)--(iii), equation \eqref{e1.1} was studied   in \cite{1} and  \cite{2}.    The present approach is a general unifying one and leads to new regularity results for \eqref{e1.1} and, in particular, implies that the solution $X$ to \eqref{e1.1} is absolutely  continuous in $t$ up to multiplication with $e^{-W(t)}$.
It should be said, however, that it applies only to stochastic differential equations which admit a formulation in the variational setting $V\subset H\subset V'$  and with coercive nonlinearities.
  In~a~few words, the method is the following. By the rescaling transformation
\begin{equation}\label{e1.6}
X(t)=e^{W(t)}y(t),\ t\ge0,\end{equation}
one formally reduces equation \eqref{e1.1} to the random differential equation
\begin{equation}\label{e1.7}
\begin{array}{l}
\displaystyle \frac{dy}{dt}\,(t)+e^{-W(t)}A(t)
(e^{W(t)}y(t))+\mu y(t)=0,\mbox{ a.e. }t\in(0,T),\vsp
y(0)=x.\end{array}
\end{equation}
 The random Cauchy problem \eqref{e1.7} will be treated as an operatorial equation in a convenient Hilbert space of stochastic processes    to be described later on. A nice feature of the approach to be developed below is that, though \eqref{e1.7} is not an equation of monotone type, it can be rewritten as an operator equation of monotone type in an appropriate space of infinite dimensional stochastic processes on $[0,T]$.
It should be emphasized that the rescaling approach to the treatment of stochastic PDE with linearly multiplicative noise was previously applied in~\cite{4}, \cite{5} and one of the main advantages of this approach is that it leads to sharp  pointwise estimates and new pathwise regularity for solutions to \eqref{e1.1}. It should be mentioned, however, that such a result cannot be proved for equations with more general Gaussian processes $\sigma(X)W$, the linearity of $\sigma$ being essential for this approach.
\paragraph{Notations.} $H^k(\calo), k=1,2,$ and $H^1_0(\calo),$ $W^{1,p}_0(\calo)$, $W^{-1,p'}(\calo)$, $1\le p\le{\infty}$, are Sobolev spaces on $\calo$ (see, e.g., \cite{6}, \cite{7}). If $U$ is a Banach space we denote by $L^p(0,T;U)$, $1\le p\le\infty$, the space of   all $L^p$-integrable $U$-valued functions on $(0,T)$. Similarly, the space $L^p((0,T)\times\Omega;U)$ is defined. By $W^{1,p}([0,T];U)$ we denote the space of absolutely continuous functions $y:[0,T]\to U$ such that $\frac{dy}{dt}\in L^p(0,T;V)$. In the following we refer to \cite{8} for notations and stan\-dard results on the theory of maximal monotone operators in Banach spaces.
\section{The main result}
\setcounter{equation}{0}
\begin{theorem}\label{t2.1} Under Hypotheses {\rm(i)--(iv)}, for each $x\in H$, equation \eqref{e1.1} has a unique solution $X$ $($in the sense of Definition {\rm\ref{d1.1}}$)$.
Moreover, the function $t\to e^{-W(t)}X(t)$ is $V'$-absolutely continuous on $[0,T]$ and
\begin{equation}\label{e2.1}
\E\int^T_0\left|e^{W(t)}\,\frac d{dt}(e^{-W(t)}X(t))\right|^{p'}_{V'}dt<{\infty}.\end{equation}
\end{theorem}

The meaning of the derivative $\frac{dy}{dt}$ will be made precise below.

As mentioned earlier, the proof strategy is to reduce \eqref{e1.1} via the trans\-for\-ma\-tion \eqref{e1.6} to the random differential equation \eqref{e1.7}, which will be treated afterwards as a deterministic evolution equation.

 A heuristic application of the It\^o formula in \eqref{e1.1} leads to (see Lemma \ref{l8.1} in the Appendix for its rigorous justification)
\begin{equation}\label{e2.1a}
dX=e^Wdy+e^WydW+\mu e^Wydt,\end{equation}
where $y$ is given by \eqref{e1.6}.  Substituting into \eqref{e1.1} yields \eqref{e1.7}, that is,
\begin{equation}\label{e2.2}
\begin{array}{l}
\displaystyle \frac{dy}{dt}+e^{-W}A(t)(e^Wy)+\mu y=0,\ t\in(0,T),\vsp
y(0)=x.\end{array}\end{equation}
\begin{definition}\label{d2.1} \rm A solution to \eqref{e2.2} is an $H$-valued $(\calf_t)_{t\ge0}$-adapted process $y=y(t)$, $t\in[0,T]$, with continuous sample paths, $V'$-absolutely continuous on $[0,T]$, $\pas$, and satisfying the following conditions
\begin{eqnarray}
&\displaystyle \sup_{t\in[0,T]}\E|e^{W(t)}y(t)|^2_H<{\infty},
\label{e2.3}
\\[2mm]
&\E\displaystyle \int^T_0\left|e^{W(t)}\,
 \frac{dy}{dt}\,(t)\right|^{p'}_{V'}dt<{\infty},
 \label{e2.5}\\[2mm]
&\begin{array}{l}
\displaystyle \frac{dy}{dt}\,(t)+e^{-W(t)}A(t)(e^{W(t)}y(t))+
\mu  y(t) =0, \mbox{ a.e. }t\in(0,T),\\[2mm]
 y(0)=x,\end{array}\label{e2.6}\\[2mm]
&\E\displaystyle \int^T_0|e^{W(t)}y(t)|^{p}_{V}dt<{\infty}.\label{e2.7}
\end{eqnarray}$($Hence $A(t)(e^Wy)\in L^{p'}({\Omega}\times(0,T);V'),$ by \eqref{e1.4}$)$.\end{definition}

In particular, it follows by \eqref{e2.4a} and \eqref{e2.5}, \eqref{e2.7} that $y:[0,T]\to V'$ is absolutely continuous $\pas$ and
\begin{equation}\label{e3.8a}
y\in L^p(0,T;V),\ \ \frac{dy}{dt}\in L^{p'}(0,T;V'),\ \pas\end{equation}
The exact meaning of the function $\zeta=\frac{dy}{dt}$ arising in \eqref{e2.5}, \eqref{e2.6}, \eqref{e3.8a} is
\begin{equation}\label{e3.8b}
y(t,\omega)=y(0,\omega)+\int^t_0\zeta(s,\omega)ds,\ \forall t\in[0,T],\ \omega\in\Omega,
\end{equation}where $e^W\zeta\in L^{p'}((0,T)\times\Omega;V')$.

We have
\begin{proposition}\label{p2.1} Under Hypotheses {\rm(i)--(iv)}, for each $x\in H$, equation \eqref{e2.2} has a unique solution $y$.\end{proposition}

We shall prove Proposition  \ref{p2.1} in Section 4 via an operatorial approach to be des\-cribed later on in Section 4. Now, we reduce the proof of Theorem \ref{t2.1} to Proposition~\ref{p2.1}.

\medskip\n{\it Proof of Theorem \ref{t2.1}.} By Lemma \ref{l8.1} in the Appendix, we know that equations \eqref{e1.1} and \eqref{e2.2} are equivalent via the rescaling transformation \eqref{e1.6} and so existence and uniqueness of a solution $X$ to \eqref{e1.1} in the sense of Definition \ref{d1.1} follows by Proposition \ref{p2.1}. As regards \eqref{e2.1}, this is a direct consequence of \eqref{e2.5}, \eqref{e3.8a}.~\rule{1,5mm}{1,5mm}

\begin{remark}\label{r2.1} {\rm We must emphasize that \eqref{e2.1} is a new and  somewhat   surprising regularity result for the solution  $X$ to the stochastic equation \eqref{e1.1}. It amounts to saying that up to multiplication  with $e^{-W}$, the process $X$ is $\pas$ absolutely continuous $V'$-valued on $[0,T]$.  We recall that the standard existence theory for equation \eqref{e1.1} under Hypotheses (i)-(iii) provides a solution $X\in L^p((0,T)\times\Omega;dt\times d\mathbb{P};V)\cap L^2(\Omega;C([0,T];H))$ only (see \cite{2}, \cite{1},~\cite{3}).}  \end{remark}

\section{An operatorial approach to equation \eqref{e2.2}}


\setcounter{equation}{0}

Without loss of generality, we may assume that $A(t)$ satisfies the strong monotonicity condition
\begin{equation}\label{e18a}
\left<A(t)u-A(t)v,u-v\right>\ge\nu|u-v|^2_H,\ \forall u,v\in V,\end{equation}
where $\nu$ is defined by \eqref{e4ay}. Indeed, it is easily seen that, by the substitution \mbox{$y\to e^{-(\nu+\delta)t}y$,} equation \eqref{e2.2} can be equivalently written as
\begin{equation*}\begin{array}{l}
\displaystyle\frac{dy}{dt}\,(t)+e^{-W(t)}\widetilde A(t) (e^{W(t)}y(t))+
\mu y(t)=0,\ t\in(0,T),\vsp
y(0)=x,\end{array}\end{equation*}
where
$$\widetilde A(t)y=e^{-(\delta+\nu)t}A(t)
(e^{(\delta+\nu)t} y )+(\delta+\nu) y.$$Then  $\wt A:V\to V'$ and $\wt A$ satisfies (ii), (iii) and \eqref{e18a}.

We note that the operator $y\to e^{-W(t)}A(t)(e^{W(t)}y)$ is not monotone in \mbox{$V\times V'$} and so the standard existence theory (see, e.g., \cite{8}, \cite{9}) is not applicable in this case. Therefore, we define new spaces   $\calh,$ $\calv$ and $\calv'$, as follows. $\calh$ is the Hilbert space of all $H$-valued $(\calf_t)_{t\ge0}-$adapted processes $y:[0,T]\to H$ such that
\begin{equation}\label{e20a}
|y|_\calh=
\left(\E\int^T_0
|e^{W(t)}y(t)|^2_Hdt\right)^{\frac12}<{\infty},
\end{equation}where $\E$ denotes the expectation in the above probability space. The space $\calh$ is endowed with the norm $|\cdot|_\calh$ coming from the scalar product \begin{equation}\label{e20aa}
\left<y,z\right>_\calh=
\E\int^T_0
\left<e^{W(t)}y(t),e^{W(t)}y(t)\right>dt.\end{equation} $\calv$ is the space of the $V$-valued $(\calf_t)_{t\ge0}$-adapted process $y:[0,T]\to V$ such that \begin{equation}\label{e20aaa}
|y|_\calv=\left(\E\int^T_0|e^{W(t)}
y(t)|^p_Vdt\right)^{\frac1p}<{\infty}.\end{equation} Clearly, $\calv$ is reflexive. $\calv'$ (the dual of $\calv$) is the space of all $V'$-valued $(\calf_t)_{t\ge0}$-adapted processes $y:[0,T]\to V'$ such that
\begin{equation}\label{e20aaaa}
|y|_{\calv'}=\left(\E\int^T_0|e^{W(t)}y(t)|^{p'}_{V'}dt
\right)^{\frac1{p'}}<{\infty},\end{equation}where $\frac1p+\frac1{p'}=1.$
If $2\le p<{\infty}$,   we have
\begin{equation}\label{e3.1}
\calv\subset\calh\subset\calv'\end{equation}continuously and densely   and

\begin{equation}\label{e3.2}
_{\calv'}\left<u,v\right>_{\calv}
=\E\int^T_0\left<e^{W(t)}u(t),
e^{W(t)}v(t)\right>dt, v\in\calv,\ u\in\calv',\end{equation} is just the duality pairing between $\calv$ and $\calv'$. We have also
\begin{equation}\label{e4.27}
_{\calv'}\left<u,v\right>_\calv=\left<u,v\right>_\calh,\ \forall  u \in\calh, \ v\in V.\end{equation}
In the case where $1<p<2$ we replace $\calv$ by $\calv\cap\calh$ and still have \eqref{e3.1}.

We also note that  we have the continuous embeddings
\begin{equation}\label{e4.8prim}
\barr{ll}
L^{p_2}((0,T)\times\Omega;V)\subset \calv\subset L^{p_1}((0,T)\times\Omega;V),&\forall 1\le p_1<p,\vsp
& \max(p,2)<p_2. \earr\end{equation}
Now, we fix $x\in H$ and define on $\calv$ the operators $\cala:\calv\to\calv'$ and $\calb:D(\calb)\subset\calv\to\calv'$ as follows:
\begin{eqnarray}
(\cala y)(t)\!&{=}&\!e^{-W(t)}A(t)(e^{W(t)} y(t))-\nu  y(t),\label{e3.3}   \mbox{ a.e. }t\in(0,T),\, y\in\calv, \\[2mm]
(\calb y)(t)\!&{=}&\!\frac{dy}{dt}\,(t)+(\mu+\nu) y(t),\ \mbox{ a.e. }t\in(0,T),\ y\in D(\calb),\nonumber\end{eqnarray}
\begin{eqnarray}
D(\calb)&{=}&\Big\{y\in\calv:y\in AC([0,T];V')\cap C([0,T];H),\ \pas,\label{e3.4}\\
  &&\qquad\qquad\qquad\qquad\qquad\qquad  \frac{dy}{dt}\in\calv',\ y(0)=x\Big\}.\nonumber
\end{eqnarray}
Here, $AC([0,T];V')$ is the space of all absolutely continuous $V'$-valued functions on $[0,T]$ and $\frac{dy}{dt}$ is defined as in \eqref{e3.8b}. The fact that indeed $\cala(\calv)\subset\calv'$ follows by \eqref{e1.4} since $p\ge p'$ if $p\ge2$, and $\calv$ is replaced by $\calv\cap\calh$ for $1< p<2$.

We also note that, if $p\ge2$ and $y\in L^p(0,T;V)$ and $\frac{dy}{dt}\in L^{p'}(0,T;V')$, that is, if $y\in W^{1,p'}([0,T];V')\cap L^p(0,T;V)$, then $y\in C([0,T];H)$  and $\frac{dy}{dt}$ is just the derivative of $y$ in the sense of $V'$-valued distributions on $(0,T)$ (see, e.g., \cite[Corollary 2.1, p. 33]{8}) and so the condition $y\in C([0,T];H)$ in the definition of $D(\calb)$ is redundant. That the same is true for $1<p<2$, follows by Lemma \ref{l92} in the Appendix.

We note also that, by virtue of \eqref{e4.8prim}, we have
$$\frac{dy}{dt}\in L^{p_1}((0,T)\times\Omega;V'),\ \forall y\in D(\calb),$$for any $1\le p_1<p'$ and so $y:[0,T]\to L^{p_1}(\Omega;V')$ is absolutely continuous, $\pas$

Then, the Cauchy problem \eqref{e2.2} can be written as the operatorial \mbox{equation}
\begin{equation}\label{e3.5}
\calb y+\cala y=0.\end{equation}

We have

\begin{lemma}\label{l3.1} The operators $\cala$ and $\calb$ are maximal monotone from $\calv$ to $\calv'$.\end{lemma}

\begin{proof} $\cala:\calv\to\calv'$ is maximal monotone by virtue of the classical Minty-Browder result (see, e.g., \cite{8}, p. 43)  because, as easily seen, by virtue of (ii)--(iv) and \eqref{e18a} it   is monotone, that is,
$$_{\calv'}\left<\cala y-\cala z,y-z\right>_{\calv}\ge0,\ \forall  y,z\in \calv,$$and   demicontinuous on $\calv$ (that is strongly-weakly continuous).

Indeed, if $y_n\to y$ in $\calv$, then Hypothesis (iii) implies that (along a subsequence) $y_n\to y\,dt\otimes \mathbb{P}\mbox{-a.e.}$ Hence $dt\otimes\mathbb{P}\mbox{-a.e.}$
\begin{eqnarray}\label{e3.6}
\begin{array}{l}
(\cala y_n)(t)=e^{-W(t)} A(t)(e^{W(t)} y_n(t) ){-}\nu y_n(t) \vsp
\qquad\quad \longrightarrow
e^{-W(t)}A(t)(e^{W(t)} y(t) )-\nu  y(t) \
 \mbox{ weakly in }V'\end{array}\end{eqnarray}
because, by Hypothesis (iv), $e^{W(t)}$ is a multiplier in $V$. We also have by \eqref{e1.4}
$$|\cala y_n|_{\calv'}\le\g_1|y_n|^{p-1}_\calv
+\g_2+\g_3|y_n|_{\calh},\ \forall  n\in\mathbb{N},$$and so, by the Banach-Alaoglou theorem, we have that (along a subsequence)
\begin{equation}\label{e3.7}
\cala y_n\to\eta\mbox{ weakly in $\calv'$}.\end{equation}On the other hand, by \eqref{e3.6} we have that, for a countable dense subset of $\psi\in V,$
\begin{equation}\label{e3.8}
\hspace*{-5mm}\barr{r}
_{V'}\!\left<\cala y_n(t),\psi\right>_{V}\to{}_{V'}
\left<e^{-W(t)}A(t)(e^{W(t)} y(t) )-\nu  y(t) ,
\psi\right>_V,\vsp dt\otimes\mathbb{P}-\mbox{a.e. on $(0,T)\times{\Omega}.$}\earr \end{equation}
By \eqref{e3.7} and \eqref{e3.8}, we see that
$$\barr{r}
\eta=e^{-W(t)} A(t)(e^{W(t)}  y(t)) -\nu  y(t) =(\cala y)(t),\  dt\otimes\mathbb{P}-\mbox{a.e. in $(0,T)\times{\Omega}$,}\earr$$as claimed.

We prove, now, that $\calb$ is maximal monotone.

To prove the monotonicity, let $y_1,y_2\in D(\calb)$     and set $y:=y_1-y_2.$ Then
\begin{equation}\label{e3.9}
\begin{array}{r}
_{\calv'}\left<\calb(y),y\right>_\calv
=\E\displaystyle \int^T_0\left<e^{W(t)}\frac d{dt}\,y(t),
e^{W(t)}(y(t))\right>dt\vsp
\qquad\qquad+\E\displaystyle \int^T_0\left<(\mu+\nu)
y e^W,y e^W\right>dt.\end{array}\end{equation}
By   Lemma  \ref{l8.1}(jjj) in the Appendix, we have, for $y\in D(\calb)$,
\begin{equation}\label{e4.15a}
d(y e^W)=e^Wdy+e^W y dW+\mu e^W y dt,\end{equation}and, applying It\^o's formula  to $|ye^W|^2_H$ (which is justified   because, e.g., if $p\ge2,$ Theo\-rem 4.2.5 in \cite{3} and, if $p\in(1,2)$, Remark 4.2.8 (iii) in \cite{18prim}, respectively, can be applied due to   Lemma \ref{l8.1} (jjj) below), we obtain
\begin{equation}\label{e419}
\begin{array}{ll}
\displaystyle\frac12\,d|e^W y|^2_H
&=\left<e^W\,\dd\frac{dy}{dt}\,,e^W y \right>dt
+\left<e^W y ,e^W y\, dW\right>\vsp
&+\left<\mu e^W y ,e^W y \right>dt
+\displaystyle\frac12
\sum^\infty_{j=1}|e^W y e_j|^2_H\mu^2_jdt.
\end{array}
\end{equation}Then, by \eqref{e4ay}, we obtain that
\begin{equation}\label{e3.10a}
\begin{array}{lcl}
\E \displaystyle \int^T_0 \left< e^W\,\frac d{dt}  y ,e^W y  \right> dt\vsp
\qquad \ge
-\E \displaystyle \int^T_0 \left<(\mu+\nu) y e^W, y e^W\right>dt +\displaystyle \frac12\,\E|e^{W(T)} y(T) |^2_H-\dd\frac 12\,|y(0)|^2_H.\end{array}\end{equation}
Substituting into \eqref{e3.9}  yields (because $y(0)=0)$
\begin{equation}\label{e3.10b}
_{\calv'}\left<\calb y,y\right>_\calv\ge \frac12\,\E|e^{W(T)}y(T)|^2_H.\end{equation}Hence, $\calb$ is monotone.
\end{proof}

  As regards the  maximality of $\calb$ in $\calv\times\calv'$, we have


\begin{lemma}\label{l8.2} Let $\calb:\calv\to\calv'$ be the operator defined by \eqref{e3.4}.
Then, $\calb$ is maximal monotone in $\calv\times\calv'$.
\end{lemma}

\begin{proof} We denote by $J:V\to V'$ the duality mapping of $V$. We recall that $|J(y)|_{V'}=|y|_V$,
$_{V'}\!\left<J(y),y\right>_V=|y|^2_V.$ Moreover, since, by assumption,  $V'$ is strictly convex,  then $J$ is single valued, monotone and demicontinuous (see \cite{8}, p. 12). We define the mapping $F:\calv\to\calv'$ as
\begin{equation}\label{e4.39az}
(Fy)(t)=e^{-W(t)}J(e^{W(t)}y(t))|e^{W(t)}y(t)|^{p-2}_V,\ \forall t\in(0,T),\ y\in\calv.\end{equation}
 Clearly, $F$ is monotone and demicontinuous from $\calv$ to $\calv'$. Furthermore, $y\to F(y)|y|^{2-p}_\calv$ is just the duality mapping of $\calv$, if $p\ge2$. If $p\in(1,2)$, we replace $F$ by $F+I$ and proceed analogously as below with $\nu+1$ replacing~$\nu$.

It suffices to show that, for each $f\in\calv'$, the equation $\calb y+F(y)=f$, where $F$ is given by \eqref{e4.39az}, has a solution $y$ (see \cite[Theorem 2.3, p.~35]{8}).

Equivalently, $y\in D(\calb)$, and
\begin{equation}\label{e8.2}
\barr{l}
\dd\frac{dy}{dt}+e^{-W} J(e^{W}y)|e^{W}y|_V^{p-2}+(\mu+\nu)y=f,\ t\in(0,T),\vsp
y(0)=x.\earr\end{equation}
We set $G(z)=J(z)|z|^{p-2}_V$, $\forall z\in V$ and
note that $G$ is monotone, demi\-con\-ti\-nuous and coercive from $V$ to $V'$. We denote by $G_H$ the restriction of $G$ to
$H$, that is, $G_H$ is the operator with graph
 $\{(u,Gu):u\in V\}\cap V\times H.$  The operator $G_H$ is maximal monotone in $H\times H$.  Let $G_\lbb$ denote its Yosida approximation, that is,
\begin{equation}\label{e8.3}
G_\lbb(z)=G(I+\lbb G_H)^{-1}(z)=\frac1\lbb\ (z-(I+\lbb G_H)^{-1}(z)),\ \forall z\in H.\end{equation}
We consider now the approximating equation of \eqref{e8.2}
\begin{equation}\label{e8.4}
\barr{l}
\dd\frac{dy_\lbb}{dt}+e^{-W}G_\lbb(e^W y_\lbb)+(\mu+\nu)y_\lbb=f,\ t\in[0,T],\vsp
y_\lbb(0)=x.\earr\end{equation}
We assume first that $f\in\calh$ and prove later on that the existence extends to all of $\calv'$. Since $G_\lbb$ is Lipschitz on $H$, it follows that \eqref{e8.3} has a unique solution $y_\lbb\in C([0,T];H)$, $\pas$ Moreover, $t\to y_\lbb(t)$ is $(\calf_t)_{t\ge0}$-adapted and by \eqref{e419}

\begin{equation}\label{e8.5}
\barr{l}
\dd\frac12\ \E|e^{W(t)}y_\lbb(t)|^2_H+\E\int^t_0\!\!\!
\left<G_\lbb(e^Wy_\lbb),e^Wy_\lbb\right>ds
+\nu\E\dd\int^t_0\!\!\!  |e^Wy_\lbb|^2_Hds\\
 =\dd\frac12\,\E\dd\int^t_0\sum^\infty_{j=1}|e^W y_\lbb e_j|^2_H\mu^2_j ds+\E\dd\int^t_0\left<fe^W,e^Wy_\lbb\right>ds+\frac12\ |x|^2_H,\\\hfill \forall t\in[0,T].\earr\!\!\!\end{equation}
Taking into account that
\begin{equation}\label{e425a}
\barr{ll}
\left<G_\lbb(z),z\right>\!\!\!&\ge(G((I+\lbb G_H)^{-1}(z)),(I+\lbb G_H)^{-1}(z))\vsp
&=|(I+\lbb G_H)^{-1}(z)|^p_V,\  \forall z\in H,\earr\end{equation}
we get, via Gronwall's lemma, for all $t\in[0,T]$,
\begin{equation}\label{e8.6}
\barr{l}
\E|e^{W(t)}y_\lbb(t)|^2_H+\E\dd\int^t_0\!\!|(I{+}\lbb G_H)^{-1}(e^Wy_\lbb(s))|^p_Vds
 \le   C(|f|^{2}_{\calh}{+}|x|^2_H).\earr\end{equation}
Hence, along a subsequence $\{\lbb\}\to0$, we have
\begin{equation}\label{e8.6a}
\barr{rcll}
e^{W} y_\lbb &\longrightarrow
& e^Wy&\mbox{weak$^*$ in }L^\infty(0,T;L^{2}(\Omega;H)),\vsp
(I+\lbb G_H)^{-1}(e^{W}y_\lbb) &\longrightarrow &z&\mbox{weakly in }L^{p}((0,T)\times \Omega;V),\vsp
G_\lbb(e^Wy_\lbb)&\longrightarrow &\eta&\mbox{weakly in }L^{p'}((0,T)\times \Omega;V'),
\earr\end{equation}
where the latter follows, since for $z\in H$, $\lbb>0$, \begin{equation}\label{e429}
|G_\lbb(z)|^{p'}_{V'}=|(I+\lbb G_H)^{-1}(z)|^p_V.\end{equation}

 \eqref{e8.3} and \eqref{e429} imply that
 $$|e^Wy_\lbb-(I+\lbb G_H)^{-1}(e^Wy_\lbb)|^{p'}_{V'}\le C_1\lbb^{p'}|(I+\lbb G_H)^{-1}(e^Wy_\lbb)|^p_V.$$
Therefore, $z=e^Wy$. Moreover, letting $\lbb\to0$ in \eqref{e8.4}, we get $\pas$
\begin{equation}\label{e8.7}
\barr{l}
\dd\frac{dy}{dt}+e^{-W}\eta+(\mu+\nu)y=f\mbox{\ \ a.e. }t\in(0,T),\vsp
y(0)=x,\earr\end{equation}
and, clearly, $y\in D(\calb)$.  It remains to show that $\eta=G(e^Wy)$. To this end, by the inequa\-lity in \eqref{e425a} and Lemma \ref{l93} below (see, also, \cite[Lemma 2.3]{8}), it suffices to show that
\begin{equation}\label{e8.8}
\limsup_{\lbb\to0}\ \E\int^T_0\int^t_0\left<G_\lbb(e^W y_\lbb),e^{W}y_\lbb  \right>ds\,dt\le
\E\dd\int^T_0\int^t_0\left<\eta,e^Wy\right>ds\,dt.\end{equation}
To this end, we note   that, by \eqref{e8.5} and \eqref{e8.6a}, \eqref{e429}, we have
$$\barr{l}
\dd\limsup_{\lbb\to0}\E\dd\int^T_0\int^t_0\left<G_\lbb(e^W y_\lbb),e^W y_\lbb\right>ds\,dt\vsp
\qquad\le\dd-\nu\E\dd\int^T_0\int^t_0|e^Wy|^2_Hds\,dt
-\dd\frac12\,\E\int^t_0|e^{W(t)}y(t)|^2_Hdt
+\dd\frac T2\,|x|^2_H\vsp
\qquad+\E\dd\int^T_0\int^t_0\left<f e^W, e^Wy\right>ds\,dt +\dd\frac12\,\E\sum^\infty_{j=1}
\int^T_0\int^t_0|e^Wye_j|^2_H\mu^2_jds\,dt,\earr$$
because both $z\to\E\int^t_0|z|^2_Hdt$  and  the function $$z\to\E\int^T_0\int^t_0\left(\nu|z|^2_H-
\frac12\sum\limits^\infty_{j=1}
|z e_j|^2_H\mu^2_j\right)ds\,dt$$ are continuous on $L^2((0,T)\times\Omega;H)$ and convex, hence   weakly lower semicontinuous in $L^2((0,T)\times\Omega;H)$.
On the other hand,   \eqref{e8.7}, Lemma \ref{l8.1} (jjj) and \eqref{e419} yield
\begin{equation}\label{e8.9}
 \barr{l}
 \dd\frac12\,\int^T_0\!\!\!
 \E|e^Wy(t)|^2_Hdt+\E\dd\int^T_0\int^t_0\!\! \left<\eta,e^Wy\right>ds\,dt\vsp
 \qquad=\dd\frac T2\ |x|^2-\nu\E\dd\int^T_0\int^t_0\!\!|e^Wy|^2_Hds\,dt\vsp
 \qquad+\dd\frac12\,\E\sum^\infty_{j=1}\int^T_0\int^t_0\!\!
 |e^Wye_j|^2_H\mu^2_j ds\,dt
 +\E\dd\int^T_0\int^t_0\!\!\left<fe^W,e^Wy\right>ds\,dt,\earr\end{equation}
and so, \eqref{e8.8} follows.

Now, let $f\in\calv'$ and choose $\{f_n\}\subset\calh$ such that $f_n\to f$ strongly in $\calv'$. If we denote by $y_n$ the corresponding solution of \eqref{e8.2}, we obtain  as in the previous case (see~\eqref{e8.5})

$$\barr{l}
\dd\frac12\,  \E|e^{W(t)}y_n(t)|^2_H+
\E\dd\int^t_0\!\!\!\left<G(e^{W(s)}y_n(s)),e^{W(s)}y_n(s)\right>ds
{+}\nu\E\dd\int^t_0\!\!|e^{W(s)}y_n(s)|^2_Hds\vsp
=\dd\frac12\ \E\!\dd\int^t_0\sum^\infty_{j=1}|e^{W(s)}y_n(s)e_j|^2_H\mu^2_jds
{+}\E\dd\int^t_0\!\!\left<f_n(s)e^{W(s)},y_n(s)e^{W(s)}\right>ds+\dd\frac12\ |x|^2_H,\earr$$
and this yields  as above
$$\E|e^{W(t)}y_n(t)|^2_H+\E\dd\int^t_0|e^{W(s)}y_n(s)|^p_Vds\le C\left(|x|^2_H+\dd\int^t_0|f_n(s)e^{W(s)}|^{p'}_{V'}ds\right)$$where $C$ is independent of $n$. Hence $\{y_n\}_n$ is bounded in $\calv\cap L^\infty((0,T)\times \Omega;H)$ and so, along a subsequence, we have
$$\barr{rcll}
e^Wy_n&\longrightarrow&e^Wy&\mbox{weak$^*$ in $L^\infty(0,T;L^2(\Omega;H))$ and}\vsp
&&&\mbox{weakly in $L^p((0,T)\times\Omega;V)$}\vsp
G(e^Wy_n)&\longrightarrow&\eta&\mbox{weakly in $L^{p'}((0,T)\otimes\Omega;V')$},\earr$$where $y$ satisfies \eqref{e8.7}. Arguing as in the proof of \eqref{e8.8}, we see that
$$\limsup_{n\to\infty}\E\dd\int^T_0\int^t_0\left<G(e^W y_n),e^W y_n\right>ds\,dt
\le \E\dd\int^T_0\int^t_0\left<\eta,e^Wy\right>ds\,dt,$$and, by Lemma \ref{l93}, this implies     that $\eta=G(e^Wy)$, and so $y$ is a solution to \eqref{e8.2}, as claimed.
This completes the proof.
\end{proof}

\bigskip\n{\it Proof of Proposition \ref{p2.1}.} Since $\cala,\calb$ are maximal monotone in $\calv\times\calv'$ and $D(\cala)=\calv$, we infer (see, e.g., \cite{8}, p.~43) that $\cala+\calb$ is maximal monotone in $\calv\times\calv'$.
Hence, the equation
\begin{equation}\label{e3.13}
{\lambda} F(y_{\lambda})+\calb y_{\lambda} +\cala y_{\lambda}=0\end{equation}
has, for each ${\lambda}>0$, a unique solution $y_{\lambda}\in D(\calb)$. (See  \cite{8}, p.~35.)

We may rewrite \eqref{e3.13} as
$$\begin{array}{l}
{\lambda} J(e^{W(t)}y_{\lambda}(t))
|e^{W(t)}y_{\lambda}(t)|^{p-2}_V+
e^{W(t)}\,
\displaystyle \frac{dy_{\lambda}(t)}{dt}
+A(t)(e^{W(t)} y_{\lambda}(t) )\vsp\hfill
+\mu e^{W(t)} y_{\lambda}(t) =0,\ t\in(0,T),\ {\omega}\in{\Omega},\vsp
y_{\lambda}(0)=x.\end{array}$$
If we apply $\left<e^{W(t)}y_\lbb (t),\cdot\right>$ to the latter and integrate over $(0,T)\times{\Omega}$, we get by \eqref{e1.3} and \eqref{e3.10a} that
$$\begin{array}{l}
(\alpha_1+{\lambda})
\E\displaystyle \int^t_0|e^{W(s)} y_{\lambda}(s) |^p_Vds
+\displaystyle \frac12\,\E|e^{W(t)}y_{\lambda}(t)|^2_H\vsp
\qquad\qquad\le\displaystyle\frac12\
 |x|^2_H+(|\alpha_2|+\nu)\displaystyle \int^t_0\E
|e^{W(s)} y_{\lambda}(s) |^2_Hds,\ \forall  t\in[0,T].\end{array}$$
By Gronwall's lemma, we see that
$$|y_{\lambda}|_\calv\le C,\ \ \forall {\lambda}>0,$$where $C$ is independent of $\lambda$ and so, along   a subsequence  again denoted by~${\lambda}$, we have
$${\lambda} F(y_{\lambda})\to 0\mbox{ strongly in }\calv', \ y_{\lambda}\to y^*\mbox{ weakly in }\calv\mbox{ as }{\lambda}\to0.$$
Hence, $\cala y_{\lambda} +\calb y_{\lambda}\to 0$ in $\calv'$ and, since $\cala+\calb$ is weakly-strongly closed (as~a~consequence of the maximal monotonicity), we have that $y^*\in D(\calb)$ and
$$\calb y^*+\cala y^*=0.$$So, $y^*$ is a solution to \eqref{e2.2} in sense of Definition \ref{d2.1}. If $y$ and $z$ are two solutions, we have by the monotonicity of $\cala$ and \eqref{e3.10b} that $$y(T)=z(T),\ \pas$$On the other hand, $y,z$ are solutions to \eqref{e2.2} on each interval $(0,t)$ and so we conclude, by the $H$-continuity of $y$ and $z$, that $\pas$, $$y(t)=z(t),\ \forall  t\in[0,T].$$This completes the proof. \hfill$\Box$

\begin{remark}\label{r3.3} {\rm The assumption:

 \begin{quote}{\it $A(t)$ single-valued and demicontinuous from $V$ to $V'$}\end{quote}

 \n can be relaxed to:

  \begin{quote}{\it For each $t\in(0,T)$, $A(t)$ is a maximal monotone $($multivalued$)$ operator from $V$ to $V'$
such that $D(A(t)){=}V,\,\forall  t\in[0,T]$ and
$A(t)(e^Wy)\cap\calv'\ne\emptyset$ for all $y\in\calv.$}\end{quote}

\n
Then, the operator $\cala$ is maximal monotone   in $\calv\times\calv'$ and, since \mbox{$D(\cala)=\calv$,} we conclude as above that $0\in R(\cala+\calb)$, as desired.
}\end{remark}

Proposition \ref{p2.1} has the following immediate extension.

\begin{proposition}\label{p2} Under the above assumptions, let $\mathcal{T}:D(\mathcal{T})\subset \mathcal{V}\to \mathcal{V}'$ be a maximal monotone operator $($possibly multivalued$)$ such that $\mathcal{B}+\mathcal{T}$ is maximal monotone on \mbox{$\mathcal{V}\times\mathcal{V}'$.} Then, for any $f\in\calv'$, there is a unique solution $y$ to the equation
\begin{equation}\label{e38a}
\mathcal{B} y+\mathcal{T} y+\mathcal{A} y=f.\end{equation}\end{proposition}

\begin{proof} Since $\calb+\mathcal{T}$ and $\cala$ are maximal monotone operators in $\calv\times\calv'$ and $\cala$ is defined on all of $\calv$, by the above mentioned Rockafellar result, $\calb+\mathcal{T}+\cala$ is maximal monotone on $\calv\times\calv'$ and so $R(\lambda F+\calb+\mathcal{T}+\cala)=\calv'$. Then, letting $\lambda\to0$ as in the above proof, we conclude that \eqref{e38a} has at least one solution $y\in D(\calb)\cap D(\mathcal{T})$. The uniqueness follows as above.~\end{proof}\smallskip

In particular, Proposition \ref{p2} applies to the finite dimensional stochastic differential equation
\begin{equation}\label{e4.26}
\begin{array}{l}
dX+F(X)dt\ni XdW,\ t\in(0,T),\vsp
X(0)=x,\end{array}\end{equation}
where $F:\rr\to\rr^d$ is a maximal monotone graph (multivalued) in \mbox{$\rr^d\times\rr^d$,} such that $D(F)=\rr^d$. Then, the operator $\mathcal{T}y=e^{-W}F(e^Wy)$ is maximal monotone in $\calh\times\calh$ (here, $H=V=V'=\rr^d)$ and $D(\calf)=\calh$. Hence $\calb+\calf$ is maximal monotone and so, for each $f\in\calh$, equation \eqref{e38a} (and, implicitly, \eqref{e4.26}) has a unique solution.
This result is, in particular, applicable to the finite dimensional stochastic differential equations \eqref{e4.26} with the monotonically nondecreasing, discontinuous function $F$ after filling the jumps in discontinuity points.

\begin{remark}\label{r4.5} {\rm Equation \eqref{e1.1} with additive noise, that is,  $dX+A(t)Xdt=dW,$  $X(0)=x$, reduces via the transformation $X=Y-W$, to the random differential equation
$$\frac{dY}{dt}+A(t)(Y+W(t))=0,\ \ Y(0)=x,$$which, under assumptions (i)--(iii), has a unique solution $Y=Y(t,\omega)$ by virtue of the standard existence theory for the Cauchy problem associated with nonlinear, monotone and demicontinuous operators $Y\to A(t)(Y+W(t))$ from $V$ to $V'$. (See \cite{8}, p. 183 and~\cite{9}.)}\end{remark}


\section{The subgradient case $A(t)=\partial\varphi(t,\cdot)$}
\setcounter{equation}{0}

Assume now that $A$  satisfies (ii), (iii) and is the subdifferential of a continuous convex function on $V$. More precisely, $A(t)=\partial\varphi(t,\cdot)$, where $\varphi(t)=\varphi(t,\omega,y)$, is measurable in $(t,\omega)$, continuous and convex in $y\in V$. That is,
$$\barr{r}
A(t)y=\{\eta\in V';\ {}_{V'}\!\!\left<\eta,y-z\right>_V\ge\varphi(t,y)-\varphi(t,z),\ \ \forall  z\in V,\ t\in(0,T),\ \pas\}.\earr$$Then, $\cala$ defined by \eqref{e3.3} is itself the subdifferential $\partial\Phi:\calv\to\calv'$ of the convex lower-semicontinuous function $\Phi:\calv\to\rr$ defined by
$$\Phi(y)=\E\int^T_0(\varphi(t,e^{W(t)}y(t))
-\nu|e^{W(t)} y(t) |^2_H)
dt,\ \forall y\in\calv.$$
Then,  taking into account that $\Phi^*(u)+\Phi(y)\ge \left<y,u\right>_{\calv'}$ for all $u\in\calv'$, \mbox{$y\in\calv$,} with equality if $u\in\partial\Phi(y)$, we may  equivalently write equation \eqref{e3.5}  as
\begin{equation}\label{e42b}
\calb y+u=0,\ \ \Phi(y)+\Phi^*(u)-{}_{\calv'}\!\!\left<u,y\right>_\calv=0,\end{equation}where $\Phi^*:\calv'\to\rr$ is the conjugate of $\Phi$, that is,
$$\Phi^*(v)=\sup\{{}_{\calv'}\!\!\left<v,u\right>_\calv-\Phi(u);\ u\in\calv\}.$$
Taking into account that
$$\Phi(z)+\Phi^*(v)-{}_{\calv'}\!\!
\left<v,z\right>_\calv\ge0,\ \forall v\in\calv',\ z\in\calv,$$it follows that the solution $y$ to equation \eqref{e3.5} (equivalently, \eqref{e2.2}) is the solution to the minimization problem
$${\rm Min}\{\Phi(y)+\Phi^*(u)-{}_{\calv'}\!\!\left<u,y\right>_\calv;\ \calb y+u=0\}.$$Equivalently,
\begin{equation}\label{e42aa}
{\rm Min}\left\{\Phi(y){+}\Phi^*({-}\calb y)+{}_{\calv'}\!
\left<\calb y,y\right>_\calv;\,  y \in  D(\calb)\right\}\!.\hspace*{-2,5mm}\end{equation}In this way, the Cauchy problem \eqref{e2.2} and, implicitly, the stochastic differential equation \eqref{e1.1} reduces to the convex minimization problem \eqref{e42aa}.

Taking into account that, under assumptions (i)--(iii), the function
$$y\to\Phi(y)+\Phi^*(-\calb y)+\frac12\,\E|e^{W(T)}y(T)|^2_H$$ is convex, lower semicontinuous and coercive on $\calv$, we infer (without invoking Proposition \ref{p2.1}) that \eqref{e42aa} has a solution which turns out to be just the solution to \eqref{e3.5}. This might be an alternative way to prove existence and uniqueness for  equation \eqref{e3.5} in this special subgradient case.

This  variational approach to \eqref{e1.1} in the subgradient case $A=\partial\varphi$ inspired by the Brezis-Ekeland principle was already developed  in \cite{13}--\cite{15} for some specific stochastic differential equations and it opens up the way to use   convex analysis methods to stochastic differential equations.

\section{Examples}
\setcounter{equation}{0}

Here, we briefly present a few classes of stochastic partial differential equations for which Theorem \ref{t2.1} is applicable. Everywhere in the following, $W$ is the Wiener process~\eqref{e1.2} satisfying \eqref{e4ay}.

\subsection{Nonlinear stochastic parabolic equations}

Consider the stochastic equation  in $\calo\subset\rr^d$
\begin{equation}\label{e4.1}
\begin{array}{l}
dX-{\rm div}(a(\nabla X))dt+\psi(X)dt=X dW
\mbox{ in }(0,T)\times\calo,\vsp
X=0\mbox{ on }(0,T)\times\partial \calo,\ \ X(0)=x\mbox{ in }\calo.\end{array}
\end{equation}Here, $a:\rr^d\to\rr^d$ is a continuous   mapping such that $a(0)=0$ and
\begin{equation}\label{e4.2}
\begin{array}{ll}
(a(r_1)-a(r_2))\cdot(r_1-r_2) \ge 0,&\forall r_1,r_2\in\rr^d,\vsp
a(r)\cdot r  \ge a_1|r|^p_d+a_2,&\forall  r\in\rr^d,\vsp
|a(r)|_d \le c_1|r|^{p-1}_d+c_2,&\forall  r\in\rr^d,
\end{array}\end{equation}where $a_1,c_1 > 0,\ p>1.$

The function $\psi:\rr\to\rr$ is continuous,  monotonically nondecreasing, $0=\psi(0)$ and
\begin{equation}\label{e4.3}
|\psi(r)|\le C(|r|^q_d+1),\ \forall  r\in\rr.\end{equation}Here, $\calo\subset\rr^d$ is a bounded open subset with smooth boundary $\partial \calo$ and $|\cdot|_d$ is the Euclidean norm of $\rr^d$.

Consider the spaces $H=L^2(\calo)$, $V=W^{1,p}_0(\calo)$, $V'=W^{-1,p'}(\calo)$ and the operator $A:V\to V'$ defined by
$$_{V'}\left<A y,\varphi\right>_V=\int_\calo(a(\nabla y)\cdot\nabla\varphi+\psi(y)\varphi)d\xi,\ \forall \varphi\in W^{1,p}_0(\calo).$$

Under assumptions \eqref{e4.2}, \eqref{e4.3}, where
\begin{equation}\label{e6.3az}
q<\frac{dp}{d-p}-1\mbox{\   if }d>p;\ \ q\in(1,\infty)\mbox{\  if }d=p,\end{equation}
and no growth condition on $\psi$ if $d\le p$, by the Sobolev-Gagliardo-Nirenberg embedding theorem (see \cite{6}, \cite{7}), it follows that $A$ satisfies assumptions (i)--(iii). As regards the Wiener process $W$, we assume here that besides \eqref{e4ay} the following condition holds
\begin{equation}\label{e6.4b}
\sum^\infty_{j=1}\mu^2_j|\nabla e_j|^2_\infty<\infty.\end{equation}

Taking into account that $\nabla (e^Wy)=e^W(y\nabla W+\nabla y)$ and by \eqref{e4ay}, \eqref{e6.4b},\break $W,\nabla W\in L^\infty (\calo)$, for all $t\ge0$, it follows by \eqref{e2.5a} that Hypothesis (iv) holds for $V=W^{1,p}_0(\calo)$ and $H=L^2(\calo)$, $Z(t)=C(|\nabla W(t)|^p_\infty+|W(t)|^p_{\infty})^{\frac1p}e^{W(t)}$, where $C$ is a positive constant. We have, therefore, by Theorem \ref{t2.1},

 \begin{corollary}\label{c} Under assumptions \eqref{e4ay}, \eqref{e4.3}, \eqref{e6.3az}, \eqref{e6.4b}, equation \eqref{e4.1} has, for\break  $x\in L^2(\calo)$, a unique  solution  $X\in L^\infty(0,T;L^2(\Omega;L^2(\calo)))\cap L^p((0,T)$ $\times\Omega;$ $W^{1,p}_0(\calo))$.   Moreover, $t\to e^{-W(t)}X(t)$ is $W^{-1,p'}(\calo)$ absolutely continuous on $[0,T)$, $\pas$\end{corollary}

\noindent Of course, the result remains true for the progressively measurable  pro\-ces\-ses \mbox{$a=a(t,r,{\omega}),$} $\psi=\psi(t,r,{\omega})$, where $a(t,\cdot),\psi(t,\cdot)$ satisfy \eqref{e4.2}, \eqref{e4.3} and the functions $t\to a(t,r)$, $t\to\psi(t,r)$ are of class~$L^{\infty}$ for each $r\in\rr$.
\begin{remark}\label{r4.2z} {\rm The case $p=1$,
which was studied in \cite{5},
is not covered, however, by the present result.
In fact, in this case, the space $V=W^{1,1}_0(\calo)$ is not reflexive and the solution $X$ to \eqref{e4.1} exists and is unique in a weak variational sense   in the space of functions with  bounded variation on $\calo$.}\end{remark}

\begin{remark}\label{r4.2} {\rm Equation \eqref{e4.1} with nonlinear boundary value conditions of the form
$$a(\nabla y)\cdot\!\!\vec{\ n} +\gamma(y)=0\mbox{\ \ on }(0,T)\times\partial\calo,$$
where $\!\!\vec{\ n}$ is the normal to $\partial\calo$, and $\gamma:\rr\to\rr$ is a monotonically increasing continuous function satisfying a growth condition \eqref{e4.3}, can be completely similarly treated in the variational setting $H=L^2(\calo)$, $V=W^{1,p}(\calo).$}\end{remark}

\begin{remark}\label{r6.4z} {\rm By Remark \ref{r3.3}, Corollary \ref{c} remains true for multivalued maximal monotone graphs $\psi$ which satisfy the growth condition \eqref{e4.3}, \eqref{e6.3az}. This more general case corresponds to equations of the form \eqref{e4.1} with va\-riable structure, that is, with discontinuous $\psi$.}\end{remark}

\medskip\n{\bf The H\"older continuity of $e^{-W}X$.}  Taking into account that $y=e^{-W}X$ is the solution to the random parabolic equation
\begin{equation}\label{e6.3a}
\barr{l}
y_t-e^{-W}{\rm div}(a(\nabla(e^Wy))){+}\mu y+e^{-W}\psi(e^Wy)=0\mbox{ in }Q_T=(0,T)\times\calo,\vsp
y=0\mbox{ on }(0,T)\times\pp\calo,\ \
y(0,\xi)=x(\xi),\ \xi\in\calo,\earr\end{equation}
one can obtain from the regularity theory for parabolic quasi-linear equations with principal part in divergence form (see \cite{Lady}) for $x\in L^\infty(\calo)$, the H\"older regularity  for solutions $y$ to \eqref{e6.3a} and hence for solutions of \eqref{e4.1}. The result we obtain here  is new for stochastic parabolic equations and  illustrates  the advantages of the method.

 In the sequel,  we assume that the above conditions on $a$ and $\psi$ are sa\-tis\-fied with  $p=2$ and, according to \eqref{e6.3az}, $q$ is taken in such a way that   $1<q<\frac{d+2}{d-2}$ if $d>2,$ $q\in (1,\infty)$ if $d=2$.

We rewrite \eqref{e6.3a} as
\begin{equation}\label{e6.3aa}
\barr{l}
y_t-{\rm div}\ \wt a(t,\xi,y,\nabla y)+\wt a_0(t,\xi,y,\nabla y)=0\mbox{ in }{\mathcal{Q}_T},\vsp
y=0\mbox{ on }(0,T)\times\pp\calo,\vsp
y(0,\xi)=x(\xi),\ \xi\in\calo,\earr\end{equation}where
$$\barr{lcl}
\wt a(t,\xi,y,\eta)&=&e^{-W(t,\xi)}a((\nabla W(t,\xi)y+\eta)e^{W(t,\xi)})\vsp
\wt a_0(t,\xi,\eta)&=&-\,e^{-W(t,\xi)}\nabla W(t,\xi)\cdot a((\nabla W(t,\xi)y+\eta)e^{W(t,\xi)})\vsp
&&+\,
\mu(\xi)y+e^{-W(t,\xi)}\psi(e^{W(t,\xi)}y).
\earr$$
Taking into account that, by \eqref{e2.5a}
$$\barr{lcll}
e^{-|W(t,\xi)|}&\ge&\gamma(\omega)>0,
&\forall(t,\xi)\in[0,T]\times\overline\calo,\ \omega\in\Omega,\vsp
e^{|W(t,\xi)|}&\le&\wt\gamma(\omega)<\infty,
&\forall(t,\xi)\in[0,T]\times\overline\calo,\earr$$
where $\gamma,\wt\gamma\in\bigcap\limits_{0<q<\infty} L^q(\Omega),$  we see that, for some $\g_1=\g_i(\omega)\in\rr,$ $i=1,...,5,$ $\g_1(\omega)>0$,  we have
$$\barr{rcl}
\wt a(t,\xi,y,\eta)\cdot\eta&\ge&\g_1|\eta|^2_d-\g_2|y|^2,\vsp
|(\wt a+\wt a_0)(t,\xi,y,\eta)|_d&\le&\g_3|\eta|_d+\g_4|y|^q+\g_5,
\earr$$for all $(t,\xi,y,\eta)\in[0,T]\times\overline\calo\times\rr\times\rr^d.$

Since by Corollary \ref{c} the solution $y$ to \eqref{e6.3aa} is in $C([0,T];L^2(\calo))\cap L^2(0,T;H^1_0(\calo)),$ it follows by \cite[Theorem 2.1, p.~425]{Lady} that,
for $x\in L^\infty(\calo)$, the solution $y$ is bounded,~i.e.,
$$|y(t,\xi)|\le M,\ \forall(t,\xi)\in(0,T)\times\calo.$$Hence, without loss of generality, we may replace in \eqref{e6.3aa} the functions $\wt a$~and~$\wt a_0$~by $\wt a^*$ and $\wt a_{(0)}$, defined below,
$$\wt a^*(t,\xi,y,\eta)=\left\{\barr{ll}
\wt a (t,\xi,y,\eta)&\ \ \ \mbox{ if }|y|\le M,\vsp
\wt a \left(t,\xi,M\,\frac y{|y|}\,,\eta\right)&\ \ \ \mbox{ if }|y|>M,\earr\right.$$
$$\wt a_{(0)}(t,\xi,y,\eta)=\left\{\barr{ll}
\wt a_{0}(t,\xi,y,\eta)&\mbox{ if }|y|\le M,\vsp
\wt a_{0}\left(t,\xi,M\,\frac y{|y|}\,,\eta\right)&\mbox{ if }|y|>M,\earr\right.$$and so,  we may assume that $\wt a$ and $\wt a_0$ satisfy Hypotheses (1.1)--(1.3) of Theorem
1.1 in \cite{Lady}, p. 419. Then, according to this theorem, for $x\in L^\infty(\calo)$, the solution $y$   belongs to the space $H^{\alpha,\frac\alpha2}(Q_T)$ of H\"older continuous functions of order $\alpha$ in $\xi\in\calo$ and order $\frac\alpha2$ in $t$ on $Q_T$ for some $\a>0$.

We have, therefore,

\begin{proposition}\label{p6.2} Let $x\in L^\infty(\calo)$ and let $a,\psi$ satisfy \eqref{e4.2}, \eqref{e4.3} with   \mbox{$p=q=2.$}  Then, the solution $X$ to equation \eqref{e4.1} satisfies
\begin{equation}\label{e6.3aaaa}e^{-W}X\in H^{\alpha,\frac\alpha2}(Q_T),\ \pas\end{equation}for some $\alpha=\alpha(\omega)\in(0,1).$

In particular, for each $t\in(0,T) $ and $\pas$, $\xi\to X(t,\xi)$ is H\"older continuous on $\calo$.\end{proposition}

\begin{remark}\label{r6.4}\rm By applying Theorem 4.1 in \cite[p.~444]{Lady},   one can obtain for the solution $y$ to \eqref{e6.3a}, and  implicitly  for \eqref{e4.1}, $L^\infty$-estimates for the gradient $\nabla y$. We omit the details.\end{remark}

\begin{remark}\label{r67} {\rm Assume further that $a_i,\psi,\frac{\pp a_i}{\pp r_j}\in C(\overline{\mathcal{Q}}_T),$ for $i,j=1,...,d,$ and that\break $x\in H^{2+\g}(\overline\calo)$, $x=0$ on $\pp\calo$ which is of class $H^{2+\g}$. Then, by Theorem 6.1 in \cite{Lady}, p.453, equation \eqref{e6.3aa} has a unique solution \mbox{$y\in H^{2+\g,1+\frac\g2}([0,T]{\times}\overline\calo)$,} where $\g$ is the H\"older exponent of $W(t)$. Moreover, in this case, the solution $y$ is continuous with respect to the initial data $x$ from $H^{2+\g}(\overline\calo)$ to $H^{2+\g,1+\frac \g2}(\overline{\mathcal{Q}}_T)$.  This implies, in particular,  that the solution $X$ to \eqref{e6.3a} is pathwise continuous with respect to $x\in H^{2+\g}(\overline\calo)$.}\end{remark}


\subsection{The porous media equation}

We first note that conditions (H1)--(H4) in \cite{3} (see p. 56) imply our conditions (i)--(iii) on $A$. Therefore,  all examples of stochastic partial differential equations from \cite{3} (see Section 4.1 therein) are covered by our results provided the noise is linear multiplicative. Nevertheless, in this subsection we present an example from \cite{3}, namely, the stochastic porous media equation, in more detail.

Consider here   the stochastic equation
\begin{equation}\label{e4.4}
\begin{array}{ll}
dX-\Delta \psi(t,\xi,X)dt=XdW&\mbox{in }(0,T)\times\calo,\vsp
X(0,\xi)=x(\xi)&\mbox{in }\calo,\vsp
\psi(t,\xi,X(t,\xi))=0&\mbox{on }(0,T)\times\times\partial \calo,\end{array}\end{equation}where $\calo$ is a bounded domain in $\rr^d$, $\psi:[0,T]\times\overline\calo\times\rr\to\rr$ is continuous, monotonically increasing in $r$, and there exist $a\in(0,\infty)$, $c\in[0,\infty)$, such that
\begin{equation}\label{e4.5}
\barr{rcll}
r\psi(t,\xi,r)&\ge&a|r|^p-c,&\forall r\in\rr,\ (t,\xi,r)\in[0,T]\times\overline\calo,\vsp
|\psi(t,\xi,r)|&\le&c(1+|r|^{p-1}),&\forall r\in\rr,\ (t,\xi,r)\in[0,T]\times\overline\calo,\earr\end{equation}where $p\in\left[\frac{2d}{d+2}\,,\infty\right)$ if $d\ge3$, and $p\in(1,\infty)$, for $d=1,2.$

By the Sobolev-Gagliardo-Nirenberg embedding theorem, we have $L^p\subset H^{-1}(\calo)$. To~write \eqref{e4.4} in the form \eqref{e1.1}, we change in this case the pivot space $H$. Namely, we take $V=L^p(\calo)$, $H=H^{-1}(\calo)$,   and $V'$ is the dual of $V$ with the pivot space $H^{-1}(\calo)$. We have, therefore, $V\subset H\subset V'$ and
$$V'=\{\theta\in\cald'(\calo);\ \theta=-\Delta v,\ v\in L^{p'}(\calo)\},$$where $\Delta$ is taken in sense of distributions.  The duality $_{V'}\left<\cdot,\cdot\right>_V$ is defined as
$$_{V'}\left<\theta,u\right>_V=\int_\calo\widetilde \theta ud\xi,\ \ \widetilde \theta=(-\Delta)^{-1}\theta.$$ $\Delta$ is the Laplace operator with homogeneous Dirichlet boundary conditions, and so\break $\widetilde \theta\in L^p(\calo)$.

The operator $A(t):V\to V'$ is defined by
$$_{V'}\left<A(t)y,v\right>_V=\int_\calo\psi(t,\xi,y)vd\xi,\ \forall  y,v\in V,\ t\in[0,T].$$
By \eqref{e4.5}, we infer that $A(t)$ satisfies \eqref{e1.3}, \eqref{e1.4}, that is,
$$\begin{array}{rcll}
_{V'}\left<A(t)y,y\right>_V&\ge&\alpha_1|y|^p_V+\alpha_2,&\forall  y\in V,\vsp
|A(t)y|_{V'}&\le&\g_1|y|^{p-1}_V+\g_2,&\forall  y\in V.\end{array}$$
(See \cite[pp.71--72]{3} for details.) It is also readily seen that $A(t):V\to V'$ is demi\-continuous.

As regards Hypothesis (iv), it is easily seen by \eqref{e4ay} that $e^{\pm W(t)}$ is~a~multiplier in $L^p(\calo)$ and $H^{-1}(\calo)$, and that \eqref{e2.4a} holds for $V=L^p(\calo)$,  \mbox{$H=H^{-1}(\calo)$} and $Z(t)=C\exp|W(t)|_\infty$ for some $C>0$.

Then, applying Theorem \ref{t2.1}, we find that:

\begin{corollary}\label{c7.3}
 For each $x\in H^{-1}(\calo)$  there is a unique solution $X\in L^{\infty}(0,T;$ $L^2({\Omega},H^{-1})$ which satisfies $X\in L^{p}((0,T)\times{\Omega};L^p(\calo))$.

 Moreover, $t\to e^{-W(t)}X(t)$ is $V'$-absolutely continuous on $[0,T]
$, $\pas$~and
\begin{equation}\label{e610}
\E\int^T_0\left|e^{W(t)}\frac d{dt}(e^{-W(t)}X(t))\right|^{p'}_{V'}dt<{\infty}.
\end{equation}
\end{corollary}

\begin{remark}\label{r4.2a} {\rm Through \eqref{e610}, Corollary \ref{c7.3} improves the corres\-pon\-ding results in \cite{3}, if the noise is linear multiplicative. In addition, Corollary \ref{c7.3} can be generalized. In fact,   we can take $V=L_d\cap H^{-1}(\calo)$, where $L_d$ is the Orlicz space on $\calo$ corresponding to a $\Delta_2$-regular $d$-function (see \cite{6}, p.~232). Also, if $e_j$ are multipliers in $H^{-1}(\rr^d)$, we can apply Theorem \ref{t2.1} to the case of the unbounded domain $\calo=\rr^d$. Then,  by Theorem \ref{t2.1}  we recover and improve many of the results of \cite{16} and \cite{10a}, provided the noise in linear multiplicative.}\end{remark}


\subsection{The stochastic transport equation}

Consider the stochastic first order hyperbolic equation

\begin{equation}\label{e38b}
\begin{array}{l}
dX(t,\xi)-\displaystyle\sum^d_{i=1}a_i(t,\xi)
\frac{\partial X(t,\xi)}{\partial \xi_i}\,dt+b(t,\xi)X(t,\xi)dt\vsp
\hfill
+\lambda|X(t,\xi)|^{p-2}X(t,\xi)dt=X(t,\xi)dW(t,\xi)\ \mbox{ in }(0,T)\times\calo,\vspace*{3mm}\\
X(0,\xi)=x(\xi),\   \xi\in\calo,\vsp
X(t)=0\mbox{ on }\Sigma=\left\{(t,\xi)\in[0,T]\times
\partial\calo;\sum\limits^d_{i=1}a_i(t,\xi)n_i(\xi)<0\right\}, \end{array}
\end{equation}where $\calo\subset\rr^d$ is, as usual,  an open and bounded subset with smooth boun\-dary $\partial\calo$,  $n=\{n_i\}^d_{i=1}$ is the normal to $\partial\calo$ and $a_i,b:[0,T]\times\overline\calo\to\rr$, $i=1,...,d$, are continuous functions with $\nabla_\xi a_i\in C([0,T]\times\overline\calo),$ $i=1,...,d$. We assume also that $\lambda>0$,  $p\ge2$, and that $x\in H^1(\calo)\cap L^p(\calo)$, $x=0$ on~$\Sigma$.

By transformation \eqref{e1.6}, we reduce \eqref{e38b} to the random differential transport equation

\begin{equation}\label{e38bb}
\begin{array}{l}
\displaystyle\frac{\partial y}{\partial t}-\sum^d_{i=1} a_i e^{-W}\frac\partial{\partial \xi_i}\,(e^Wy)
+(b+\mu)y+\lambda e^{(p-2)W}|y|^{p-2}y=0\vspace*{-2mm}\\\hfill\mbox{ in }(0,T)\times\calo,\vsp
y(0,\xi)=x(\xi),\ \xi\in\calo,\qquad
y=0\mbox{ on }\Sigma.\end{array}\end{equation}

We consider now the spaces $\calv,\calh,\calv'$ defined in Section 4, where $V=L^p(\calo)$,\break  $H=L^2(\calo)$, $V'=L^{p'}(\calo)$ and set
$$\barr{lcl}
\displaystyle\cala y&=&\lambda e^{(p-2)W}|y+x|^{p-2}(y+x)-\dd\sum^d_{i=1}a_ie^{-W}\frac\pp{\pp\xi_i}\,(e^Wx)+(b+\mu)x,\ \ \  y\in\calv,\\
\displaystyle\calb y&=& \displaystyle\frac{dy}{dt}+(\mu+\nu)y,\ \ y\in D(\calb),\earr $$where we used the transformation $y\to y-x$. $D(\calb)$ is given by \eqref{e3.4} with $x$ replaced by $0$, and $\nu$ by  \eqref{e4ay}. Then \eqref{e38bb} can be  equivalently written as
$$\calb y+\calt_0(y)+\cala y=0,$$where $\calt_0:\calv\to\calv'$ is given by
$$\begin{array}{lcl}
\mathcal{T}_0(y)&=&
-\displaystyle\sum^d_{i=1}
a_ie^{-W}\frac\partial{\partial \xi_i}
\,(e^W y )+(b-\nu)y,\ \forall y\in D(\mathcal{T}_0),\vsp
D(\mathcal{T}_0)&=&\{y\in\calv;\ e^{-W}a\cdot\nabla(e^Wy)\in\mathcal{V}',\ y =0\mbox{ on }\Sigma\}.\end{array}$$
The trace of $y$ on $\Sigma$ is taken in a weak distributional sense (see \cite{18}).

We assume that
\begin{equation}\label{e38bbb}\frac12\,{\rm div}_\xi\,a(t,\xi)+b(t,\xi)>\nu,\ \forall(t,\xi)\in[0,T]\times\overline\calo,\end{equation}
where $a=\{a_i\}^d_{i=1}$. Then, as is  easily seen, $\mathcal{T}_0$ is monotone. It should be said  that Proposition \ref{p2} is not directly applicable here because $\mathcal{T}_0+\calb$ is not maximal monotone. However, by \cite{18} (see also \cite[p.~330]{9} ), the operator $\calb+\mathcal{T}_0$ is closable in $L^p((0,T)\times\calo)$ for fixed $\omega\in\Omega$ and its closure $L$ is maximal monotone  in $L^p((0,T)\times\calo)\times L^{p'}((0,T)\times\calo)$. This implies that, for each fixed $\omega\in\Omega$, the equation
$$Ly+e^{-W}F(e^{-W}y)=0\mbox{\ \ in }(0,T)\times\calo,$$
has a solution $y\in D(L)$ and so, arguing as in the proof of  Lemma \ref{l8.2}, it follows that the closure $\overline{\calb+\calt_0}$ of $\calb+\calt_0$ in $\calv\times\calv'$ is maximal monotone.

We also note that Hypothesis (iv) can be checked in this case as in   Example 6.2.

Then, by Proposition \ref{p2}, the equation $$\overline{\calb+\calt_0}(y)+\cala y= 0$$ has a unique solution $y\in D(\overline{\calb+\calt_0})$.

In other words, there is a sequence  $\{y_n\}\subset D(\calb+\calt_0)$ such that $y_n\to y$ in $\calv$ and
\begin{equation}\label{e6.9}
\barr{l}
(\calb+\calt_0)(y_n)+\cala y_n\to 0\ \ \mbox{ in }\calv'.\earr
\end{equation}
We call such $y$ a generalized solution to   \eqref{e38bb} and the corresponding  \mbox{$X=e^Wy$} is a generalized solution to \eqref{e38b} in the above sense.

\begin{remark}\label{r6.7} {\rm Similarly, one can treat    stochastic equations of type  \eqref{e38b} with an integral transport term $\int_{\wt\Omega} K(t,x,z,z')X(t,z;z')dz'$ (see, e.g., \cite{9}, p.~346).}\end{remark}


\section{Extension to more general multiplicative noise}
\setcounter{equation}{0}

In this section, we indicate how the
  rescaling approach developed in the pre\-vious sections extends mutatis-mutandis to stochastic equations of the~form
\begin{equation}\label{e9.1}
\barr{l}
dX+A(t)X\,dt=\dd\sum^m_{k=1}\sigma_k(X)d\beta_k(t),\ t\in[0,T],\\
X(0)=x,\earr\end{equation}where $\sigma_k:D(\sigma_k)\subset H\to H$ are linear generators of $C_0$ mutually commuting groups on $H$. In this case, \eqref{e9.1}   reduces to a random differential equation in the space $H$. (See~\cite{17}, p.~203.) Namely, via the transformation $X=U(t)y$, where $U(t)=\prod^m_{k=1}e^{\sigma_k(\beta_k(t))}$, \eqref{e1.1} reduces~to the random differential equation
\begin{equation}\label{e9.2}
\begin{array}{l}
\displaystyle\frac{dy}{dt}+U^{-1}(t)A(t)(U(t)y)+\frac12\sum^m_{k=1}U^{-1}(t)
\sigma^2_k(U(t)y)=0,\ t\in(0,T),\\
y(0)=x.\end{array}\end{equation}Here $e^{t\sigma_k}$ is the global flow on $H$ generated by $\sigma_k$.

The existence in \eqref{e9.2} follows as in the previous case, by taking $$\cala y=U^{-1}(t)A(t)(U(t)y)$$ and replacing $\calh$, $\calv$ by the spaces of $(\calf_t)_{t\ge0}-$adapted processes $y$ on $(0,T)$ such that $$\mbox{$\E\int^T_0|U(t)y(t)|^2_Hdt<\infty$ (and, respectively, $\E\int^T_0|U(t)y(t)|^p_Vdt<\infty$).}$$

We assume that $U$ satisfies Hypothesis (iv) with $U(t)$ instead of $e^{W(t)}$ and take on $\calh$ the scalar product\
$$\left<u,v\right>=\E\int^T_0\left<U(t)u,U(t)v\right>dt,$$which extends to the duality pairing
 $_{\calv'}\left<u,v\right>_\calv\mbox{ on }\calv\times\calv'.$ Moreover, the space $\calv$ is reflexive and $\calv\subset\calh\subset\calv'$. If $A(t):V\to V'$ satisfies assumptions \mbox{(i)--(iii),} then clearly  $\cala:\calv\to\calv'$ is monotone, demicontinuous and coercive, i.e., satisfies all properties of operator \eqref{e3.3}. If $\calb$ is defined as in \eqref{e3.4}, we rewrite \eqref{e8.2} as $\calb y+\cala y=0$. Moreover, Lemma \ref{l8.2} remains true in the present situation. The proof is completely similar, considering instead of \eqref{e9.2} the equation
$$\barr{l}
\dd\frac{dy}{dt}+U^{-1}(t)J(U(t)y)|U(t)y|^{p-2}_V+(\mu+\nu)y=f\vsp
y(0)=x,\earr$$for $f\in\calv'$.
Then we conclude, as in the proof of Proposition \ref{p2.1}, that the range of $\calb+\cala$ is all of $\calv'$ and, in particular, that equation \eqref{e9.2} has a unique $(\calf_t)_{t\ge0}$-adapted solution $y:[0,T]\to H$, such that
\begin{eqnarray}
\E\sup_{t\in[0,T]}|U(t)y(t)|^2_H&<&\infty\label{e9.3}\\[2mm]
\E\int^T_0\left|U(t)\ \frac{dy}{dt}\right|^{p'}_{V'}dt&<&\infty.\label{e9.4}\end{eqnarray}
Then, Theorem \ref{t2.1} remains valid in the present case with $U(t)$ instead of~$e^{W(t)}$.

Consider as an example the nonlinear diffusion equation (see \cite{15}, \cite{12a})

\begin {equation}\label{e9.5}
\barr{l}
dX-{\rm div}(a(\nabla X))dt-\dd\frac12\ b\cdot\nabla(b\cdot\nabla X)dt
=b\cdot\nabla X\,d\beta,\
  t\in[0,T],\ \xi\in\calo,\vsp
X(0,\xi)=x(\xi),\ \xi\in\calo;\ \ X=0\mbox{ on }(0,T)\times\pp\calo.\earr\end{equation}
Here, $\calo\subset\rr^d$, $d=1,2,3,$ is a bounded open domain with smooth boundary, $\beta$ is a Brownian motion, $a:\rr^d\to\rr^d$ is a monotone mapping satisfying \eqref{e4.2} and $b\in(C^1(\overline\calo))^d$ satisfies:  ${\rm div}\ b=0$, $b\cdot\nu=0$ on $\pp\calo$ $(\nu$ is the normal to $\pp\calo)$. It should be mentioned that \eqref{e9.5} is equivalent to the Stratonovich stochastic equation
$$\barr{l}
dX-{\rm div}(a(\nabla X))dt=(b\cdot\nabla X)\circ d\beta\mbox{ in }(0,T)\times\calo,\vsp
X(0)=x,\ X=0\mbox{ on }(0,T)\times\pp\calo.\earr$$
Equation \eqref{e9.5} is of the form \eqref{e9.1}, where $H=L^2(\calo)$, $V=W^{1,p}_0(\calo)$, $V'=W^{1,p'}(\calo)$,~and
$$(U(t)f)(\xi)=\exp(\beta(t)b)(f)(\xi)=f(Z(\beta(t),\xi)),\ t\ge0,\ \xi\in\calo,$$where $Z=Z(s)$ is the flow of diffeomorphisms on $\overline\calo$ generated by the Cauchy problem
$$\frac{dZ}{dt}=b(Z),\ s\ge0;\ \ Z(0)=\xi\in\overline\calo.$$(See \cite{15}, \cite{12a}.)
 Then, \eqref{e9.2} is, in this case,
\begin{equation}\label{e9.6}
\barr{l}
\dd\frac{\pp y}{\pp t}\ (t,\xi)-U^{-1}(t)({\rm div}(a(\nabla_\xi(U(t)y(t))))=0,\ t\in[0,T],\ \xi\in\calo,\vsp
y(0,\xi)=x(\xi),\ y=0\mbox{ on }(0,T)\times\pp\calo,\earr\end{equation}and, as seen above, it follows the existence of a solution $y:[0,T]\to L^2(\calo)$ satisfying \eqref{e9.3}--\eqref{e9.4}. More details will be contained in  forthcoming work.

\section{Appendix}
\setcounter{equation}{0}

\begin{lemma}\label{l8.1}\
\begin{itemize}
\item[\rm(j)] Let $y=y(t),\ t\in[0,T]$, be an $H$-valued $(\calf_t)_{t\ge0}$-adapted process with continuous sample paths, $V'$-absolutely continuous on $(0,T)$ and satisfying \eqref{e2.3}--\eqref{e2.7}. Then, $X=e^Wy$ is a strong solution to \eqref{e1.1}, which satisfies \eqref{e1.3a}--\eqref{e1.5}.
\item[\rm(jj)] Let $X=X(t),\ t\in[0,T]$, be an $H$-valued $(\calf_t)_{t\ge0}$-adapted process with con\-ti\-nuous sample paths satisfying \eqref{e1.3a}--\eqref{e1.5}. Then, $y=e^{-W}X$ sa\-tis\-fies \eqref{e2.3}--\eqref{e2.7}.
    \item[\rm(jjj)] Let $y\in D(\calb)$ $($see \eqref{e3.4}$)$. Then $t\mapsto y(t) e^{W(t)}\in H$ is continuous $\pas$ and, $\forall t\in[0,T],$

        $\barr{l}
        \dd y(t)e^{W(t)}{=}y(0)+\dd\int^t_0\!\left[ e^{W(s)} \dd\frac{dy}{ds}+\mu e^{W(s)}y(s)\right]ds+\!\dd\int^t_0 \! y(s)e^{W(s)} dW(s).\earr$
\end{itemize}
\end{lemma}

\begin{proof} We shall follow an argument from \cite{5}. We first note that, by the  last item in \eqref{e25prim}, it follows that with $y$ also $t\mapsto e^{W(t)}y(t)\in H$ is continuous. We shall only   prove (j).   (jj) and (jjj) follow  by the same arguments. If $\{f_j\}$ is an orthonormal basis in $H$ such that $f_j\in V,$ $\forall  j$, we have, for each ${\varphi}\in V$, $t\in[0,T]$,

$$\left<{\varphi},e^{W(t)}y(t)\right>=\sum^{\infty}_{j=1}\left<e^{W(t)}\varphi,f_j\right>
\left<f_j,y(t)\right>$$ and so, by It\^o's formula, we have
$$e^{W(t,\xi)}=1+\int^t_0 e^{W(s,\xi)}dW(s,\xi)+\mu(\xi)\int^t_0W(s,\xi)ds,\ \forall  t\in[0,T],\ \forall \xi\in\calo.$$
This yields, via the stochastic Fubini theorem,
$$\begin{array}{lcl}
\left<e^{W(t)}{\varphi},f_j\right>
&\!\!\!=\!\!\!&\left<{\varphi},f_j\right>+
\displaystyle \sum^m_{k=1}\left<{\varphi} e_k f_j,\int^t_0 e^{W(s,\zeta)}d\beta_k(s)\right> +\!\displaystyle \int^t_0\!\!\left<\mu e^{W(s)}{\varphi},f_j\right>ds\vsp
&\!\!\!=\!\!\!&\left<{\varphi},f_j\right>+\displaystyle \sum^\infty_{k=1}\int^t_0
\left<e_ke^{W(s)}{\varphi},f_j\right>d\beta_k(s)
 +\displaystyle \int^t_0\left<\mu e^{W(s)}{\varphi},f_j\right>ds.\end{array}$$
Here, $y_{\varepsilon}=J_{\varepsilon}(y)$ and $J_\varepsilon$ is a   mollifier operator $J_{\varepsilon}\in L(V',H)$   chosen in a
such a way that $\displaystyle \lim_{{\varepsilon}\to0}J_{\varepsilon}(y)=y$ in $V'$. (Such a family of mappings $J_{\varepsilon}$ always exists.) Then, we see by \eqref{e2.6} that
\begin{equation}\label{e5.1}
\begin{array}{l}
\displaystyle \frac{dy_{\varepsilon}}{dt}
+J_{\varepsilon}(e^{-W}\eta)+J_{\varepsilon}(\mu y)=0,\ \mbox{a.e. }t\in(0,T),\vsp
y_{\varepsilon}(0)=J_{\varepsilon}(x),\end{array}
\end{equation}
that is,
$$y_\vp(t)+\int^t_0(J_\vp(e^{-W(s)}\eta(s))
+J_\vp(\mu y(s)))ds=J_\vp(x),\ \forall t\in[0,T],$$
where $\eta=A(t)(e^{W(t)}y(t)).$
Since $t\to\left<f_j,y_{\varepsilon}(t)\right>$ is absolutely continuous, we can apply the It\^o pro\-duct rule to get, by \eqref{e5.1}, that
$$\begin{array}{lcl}
\left<e^{W(t)}{\varphi},f_j\right>
\left<f_j,y_\vp(t)\right>&=&
 \displaystyle
\left<{\varphi},f_j\right>_H\left<f_j,x\right>\vsp
& -&\displaystyle \int^t_0\left<e^{W(s)}{\varphi},f_j\right>
\left<f_j,J_{\varepsilon}(e^{-W(s)}\eta(s)
+\mu y(s))\right>ds\vsp
&+&\displaystyle \sum^\infty_{_k=1}\int^t_0\left<f_j,y_{\varepsilon}(s)\right>
\left<e_ke^{W(s)}{\varphi},f_j\right>d\beta_k(s) \vsp&+&\displaystyle \int^t_0\left<f_j,y_\vp(s)\right>_H
\left<\mu e^{W(s)}{\varphi},f_j\right>_Hds,\ \forall  j\in\mathbb{N},\end{array}$$where $\eta=Ay.$
Summing up and interchanging the sum with the integrals, we obtain $\pas$, for all $t\in[0,T]$ and all ${\varphi}\in V$,
$$\begin{array}{r}
\left<{\varphi},e^{W(t)}y_{\varepsilon}(t)\right>
=\left<e^{W(s)}{\varphi},x\right>-\displaystyle \int^t_0\left<e^{W(s)}{\varphi},
J_{\varepsilon}(e^{-W(s)}\eta(s)+\mu y(s))\right>ds\\
 +\displaystyle \sum^\infty_{k=1}\int^t_0\left<{\varphi},
 e_ke^{W(s)}y_{\varepsilon}(s)\right>d\beta_k(s)
+\displaystyle \int^t_0\left<\mu e^{W(s)}\varphi(s),y_{\varepsilon}(s)\right>ds.\end{array}$$
Letting ${\varepsilon}\to0$, we see that $X=e^Wy$ satisfies \eqref{e1.4a}. It is also clear that  $X$ satisfies all the conditions in Definition \ref{d1.1}.\end{proof}


\begin{lemma}\label{l92} Let $1<p<2$ and $y\in L^p(0,T;V)\cap L^2(0,T;H)$ such that $\frac{dy}{dt}\in L^{p'}(0,T;V')+L^2(0,T;H)$. Then $y\in C([0,T];H)$ and
$$\frac12\,|y(t)|^2_H=\frac12\,|x|^2_H+\int^t_0\left<\frac{dy}{ds}\,,y\right>ds,\ \forall t\in[0,T].$$
\end{lemma}

\begin{proof} We have $\left(\frac{dy}{dt}\mbox{ is taken in $\cald'(0,T)$}\right)$
$$\frac{dy}{dt}=f_1+f_2,\ f_1\in L^{p'}(0,T;V'),\ f_2\in L^2(0,T;H).$$
We consider a mollifier $\rho_\vp\in C^\infty$ and set
$$y_\vp(t)=(y*\rho_\vp)(t),\ f^1_\vp=f_1*\rho_\vp,\ f^2_\vp=f_2*\rho_\vp.$$
We have $y_\vp\in C^1([0,T];V)$ and $\frac{dy_\vp}{dt}=f^1_\vp+f^2_\vp$. This yields
$$\barr{lcl}
\dd\frac12\,|y_\vp(t)-y_{\vp'}(t)|^2_H
&=&\dd\frac12\,|x|^2_H+\int^t_0
\left<f^1_\vp(s)-f^1_{\vp'}(s),y_\vp(s)
-y_{\vp'}(s)\right>ds\vsp
&&+
\dd\int^t_0\left<f^2_\vp(s)-f^2_{\vp'}(s),
y_\vp(s)-y_{\vp'}(s)\right>ds.\earr$$
Since $$\barr{lcll}
f^1_\vp&\longrightarrow&f^1&\mbox{ in }L^{p'}(0,T;V'),\vsp
f^2_\vp&\longrightarrow&f^2&\mbox{ in }L^{2}(0,T;H),\vsp
y_\vp&\longrightarrow&y&\mbox{ in }L^{p}(0,T;V)\cap L^2(0,T;H),\earr$$
as $\vp\to0$, we get that
$$\barr{lcll}
y_\vp&\longrightarrow&y&\mbox{ in }C([0,T];H)\mbox{ as }\vp\to0,\earr$$and that
$$\dd\frac12\,|y(t)|^2_H=\dd\frac12\,|x|^2_H+\dd\int^t_0\left<f^1+f^2,y\right>ds,\ \forall t\in[0,T],$$as claimed.
\end{proof}

\begin{lemma}\label{l93} Let $G(z)=J(z)|z|^{p-2}_V$ and $\{z_n\}\subset L^p((0,T)\times\Omega;V)$ be such that $z_n\to z$ weakly in $L^p((0,T)\times\Omega;V)$ and $G(z_n)\to\eta$ weakly in $L^{p'}((0,T)\times\Omega;V')$. Assume that
\begin{equation}\label{e92}
\limsup_{n\to\infty}\E\int^T_0(T-t)\left<
G(z_n(t)),z_n(t)\right>dt\le\E\int^T_0(T-t)\left<\eta(t),z(t)\right>dt.\end{equation}Then $\eta=G(z)$ a.e. in $(0,T)\times\Omega.$
\end{lemma}

\begin{proof} We set $\Phi(z)=\frac1p\,|z|^p_V$ and note that $\Phi$ is G\^ateaux differentiable and $\nabla\Phi=G$. Since $\Phi$ is convex and continuous on $V$, we have by Fatou's lemma
\begin{equation}\label{e93}
\liminf_{n\to\infty}\E\int^T_0(T-t)\Phi(z_n(t))dt\ge\E\int^T_0(T-t)\Phi(z(t))dt.\end{equation}
We also have, for $u\in L^p((0,T)\times\Omega;V)$,
$$\E\int^T_0(T-t)\left<G(z_n(t)),z_n(t)-u(t)\right>dt
\ge\E\int^T_0(T-t)(\Phi(z_n(t))-\Phi(u(t))dt.$$
Then, letting $n\to\infty$, we get by \eqref{e92}--\eqref{e93}
$$\E\int^T_0(T-t) (\Phi(z(t))-\Phi(u(t)))dt\le
\E\int^T_0(T-t)\left<\eta(t),z(t)-u(t)\right>dt.$$
Taking $u=z+\lbb v$, $v\in L^p((0,T)\times\Omega;V)$, dividing by $\lbb$ and letting $\lbb\to0$, we obtain, since $\nabla\Phi=G$,
$$\E\int^T_0(T-t)\left<G(z(t))-\eta(t),v(t)\right>dt\le0,\ \forall v\in L^p((0,T)\times\Omega;V),$$where we
used the elementary inequality
$$\frac1p\ a^p-\frac1p\,b^p\le\max(a^{p-1},b^{p-1})|a-b|,\ a,b\in[0,\infty),$$ to justify the interchange of $\lim\limits_{\lbb\to0}$  with the integrals. Hence $\eta=G$ a.e. in $(0,T)\times\Omega$, as claimed.
\end{proof}

\bigskip
\footnotesize
\noindent\textit{Acknowledgments.}
 Financial support through the SFB 701 at Bielefeld University and NSF -- Grant 0606615 is gratefully acknowledged. Viorel Barbu  was also partially supported by a grant of CNCS-UEFISCDI (Romania), project PN-II-PCE-2013.

\end{document}